\documentclass[11pt,reqno]{amsart}\usepackage{amsmath,amscd,amstext,amsthm,amsfonts,latexsym}
\usepackage[dvips]{graphicx}
\usepackage{graphicx, graphics}

\usepackage{a4wide}
\usepackage{amsmath,amssymb,amsthm,mathrsfs}
\usepackage{color}
\usepackage{enumerate}
\usepackage{enumitem}
\theoremstyle{plain}
\newtheorem{theorem}{Theorem}[section]

\theoremstyle{definition}
\newtheorem{remark}[theorem]{Remark}
\newtheorem{example}[theorem]{Example}
\newtheorem{definition}[theorem]{Definition}

\newtheorem*{condition}{Condition}

\newcommand{\field}[1]{\mathbb{#1}}

\newcommand{\NN}{\field{N}}

\usepackage[T2A,T1]{fontenc}
\DeclareSymbolFont{cyrillic}{T2A}{cmr}{m}{n}
\DeclareMathSymbol{\D}{\mathalpha}{cyrillic}{196}

\newcommand{\N}{\mathbb{N}}

\newcommand{\R}{\mathbb{R}}

\def\I{\ensuremath{{\bf 1}}}

\newcommand{\dist}{\operatorname{{dist}}}

\def\R{\ensuremath{\mathbb R}}
\def\N{\ensuremath{\mathbb N}}

\def\I{\ensuremath{{\bf 1}}}
\def\e{{\ensuremath{\rm e}}}

\def\B{\ensuremath{\mathcal B}}

\def\l{{\rm Leb}}

\def\p{\ensuremath{\mathbb P}}

\def\AA{\ensuremath{\mathcal A}}

\def\X{\mathcal{X}}
\def\ie{{\em i.e.}, }

\def\dist{\ensuremath{\text{dist}}}

\def\eps{\varepsilon}

\def\cv{\ensuremath{\text {Cor}}}
\newcommand{\dif}{\mathrm{d}}

\author[D. Azevedo]{Davide Azevedo}
\address{Davide Azevedo\\Instytut Matematyczny Polskiej Akademii Nauk, ul. \'Sniadeckich 8\\ 00-656 Warszawa\\Polska}
\email {davidemsa@gmail.com}

\author[A. C. M. Freitas]{Ana Cristina Moreira Freitas}
\address{Ana Cristina Moreira Freitas\\ Centro de Matem\'{a}tica \&
Faculdade de Economia da Universidade do Porto\\ Rua Dr. Roberto Frias \\
4200-464 Porto\\ Portugal} \email{amoreira@fep.up.pt}

\author[J. M. Freitas]{Jorge Milhazes Freitas}
\address{Jorge Milhazes Freitas\\ Centro de Matem\'{a}tica da Universidade do Porto\\ Rua do
Campo Alegre 687\\ 4169-007 Porto\\ Portugal}
\email{jmfreita@fc.up.pt}
\urladdr{http://www.fc.up.pt/pessoas/jmfreita}

\author[F.B. Rodrigues]{Fagner B. Rodrigues}
\address{Fagner Bernardini Rodrigues\\Instituto de Matem\'atica - Universidade Federal do Rio Grande do Sul\\
Av. Bento Gonçalves, 9500 - Prédio 43-111 - Agronomia\\
Caixa Postal 15080\
91509-900 Porto Alegre - RS - Brasil
} \email{fagnerbernardini@gmail.com }


\begin{document}

\title{Extreme Value Laws for dynamical systems with countable extremal sets}
\date{\today}


\maketitle
\begin{abstract}
We consider stationary stochastic processes arising from dynamical systems by evaluating a given observable along the orbits of the system. We focus on the extremal behaviour of the process, which is related to the entrance in certain regions of the phase space, which correspond to neighbourhoods of the maximal set $\mathcal M$, \ie the set of points where the observable is maximised. The main novelty here is the fact that we consider that the set $\mathcal M$ may have a countable number of points, which are associated by belonging to the orbit of a certain point, and may have accumulation points. In order to prove the existence of distributional limits and study the intensity of clustering, given by the Extremal Index, we generalise the conditions previously introduced in \cite{FFT12,FFT15}.   
\end{abstract}

\section{Introduction}

The study of Extreme Value Laws (EVL) for dynamical systems has received a lot of attention in the past few years. We refer to the recently published book \cite{LFFF16}, which makes an introduction to the subject and gives the latest developments in the field.

The main goal of this theory is to study the extremal properties of stochastic processes generated by the orbits of a dynamical system. To be more specific, we consider an observable function $\varphi$, choos an initial condition, let the system evolve and keep record of the values of this observable evaluated on the successive states that the system presents, along the evolution of this particular orbit. Then, we analyse such realisations of the stochastic process, which depend on the initial condition, in terms of their extremal behaviour, where we are particularly interested in the occurrence of abnormally high observations, exceeding high thresholds.

For chaotic systems, the sensitivity to initial conditions lends the stochastic processes just described an erratic behaviour that resembles the pure randomness of independent and identically distributed (iid) sequences of random variables studied in classical Extreme Value Theory.

However, the existence of periodic points was seen to create a strong dependence that turned out to be responsible for the appearance of clustering of exceedances of high thresholds.
In this setting the rare events corresponding to the exceedances of high thresholds correspond to the entrances of the orbit in certain regions of the phase space, where the observable function $\varphi$ is maximised. This connection between excedances and visits to certain target sets of the phase space is the base of the formal link established in \cite{FFT10,FFT11} between the existence of EVL and Hitting Time Statistics (HTS). By EVL we mean the distributional limit of the partial maxima of the considered stochastic processes and by HTS we mean the distributional limit for waiting time before hitting a certain target of the phase space.

In the literature of both EVL and HTS for dynamical systems, in most situations $\varphi$ is maximised on a certain single point $\zeta$ so that exceedances of increasingly high thresholds correspond to hits to shrinking neighbourhoods around $\zeta$. At first,  these neighbourhoods were dynamical cylinders and more recently they have been taken, more generally, as metric balls.

For hyperbolic and non-uniformly hyperbolic systems (including intermittent maps, multimodal quadratic maps, H\'enon maps, billiards) and for typical points $\zeta$, \ie for almost all $\zeta$ with respect to an invariant measure Sinai-Ruelle-Bowen (SRB) measure, the limiting laws that apply for both EVL and HTS have been proved to be the standard exponential distribution (with mean 1), see for example  \cite{HSV99, C01, BSTV03, HNT12, CC13,PS14}.

When $\zeta$ is a periodic point, then periodicity creates a short recurrence dependence structure which is responsible for the appearance of clustering of exceedances or hits to the neighbourhoods of $\zeta$ (usually, metric balls around $\zeta$). This means that the limiting law is exponential with a parameter $0\leq\theta<1$ (with mean $\theta^{-1}$). This parameter is called the Extremal Index (EI) and measures the intensity of clustering since, in most situations, $\theta^{-1}$ coincides with the average size of the cluster, \ie the average number of exceedances within a cluster. The existence of an EI was proved for hyperbolic systems, $\phi$-mixing systems, intermittent maps, Benedicks-Carleson quadratic maps, see for example: \cite{H93, A04, HV09, FFT12, FFT13}.

Moreover, in \cite{FFT12,FP12}, for uniformly expanding systems such as the doubling map, a dichotomy was shown which states that either $\zeta$ is periodic and we have an EI with a very precise formula depending on the expansion rate at $\zeta$, or for every non-periodic $\zeta$, we have an EI equal to 1 (which means no clustering). The dichotomy was obtained for more general systems such as: conformal repellers \cite{FP12}, systems with spectral gaps for the Perron-Frobenius operator \cite{K12}, mixing countable alphabet shifts \cite{KR14}, systems with strong decay of correlations \cite{AFV15} and intermittent maps \cite{FFTV15}.

Very recently, in \cite{AFFR16}, the authors considered the possibility of having a finite number of maximal points of $\varphi$. Moreover, it was shown that when these maximal points are correlated, in the sense of belonging to the same orbit of some point (not necessarily periodic), then clustering of exceedances is created by what turned out to be a mechanism that emulates some sort of fake periodic effect. The fact that the maximal points lay on the same orbit makes the occurrence of an excedance, which corresponds to an entrance in a neighbourhood of one of such points, followed by another exceedance, in a very short time period, a very likely event, which was the same effect observed at maximal periodic points. 

In this paper, we develop the theory further by letting the number of points where $\varphi$ is maximised to be infinite. Namely, we will consider countably many maximal points, which, in compact phase spaces as we consider here, have accumulation points which complicate the analysis. The tools used in \cite{AFFR16} were originally developed in \cite{FFT12} and later refined in \cite{FFT15}. They are based on some conditions on the dependence structure of the stochastic processes. 
These conditions are ultimately verified on account of the rates of decay of correlations of the systems. However, the conditions used there build upon the fact that the periodic or fake periodic effect creating the clustering has finite range. This derives from the fact that the period of both the periodic or fake periodic effect is finite since, in the first case, each periodic point has a finite period and, in the second case, the number of maximal points used to emulate periodicity is finite.  This means we need to make some adjustments to the conditions devised in \cite{FFT12,FFT15} to cope with infinite ``periods''. Another technical problem arises in the use of decay of correlations to prove the dependence conditions. Typically, the test functions plugged into the decay of correlations statements have finitely many connected components, which is not necessarily the case here. Hence, we need to play with adapted truncated function approximations that have to be chosen to balance the estimates and obtain the result.

In the recent paper, \cite{MP15}, the authors provide a very insightful numerical study that shows evidence that EVL can be proved for dynamical systems with observable functions maximised on Cantor sets. The techniques we introduce here can be used to provide a theoretical proof of some of the statements in \cite{MP15}. This is an ongoing work of the last three authors.

\section{The setting}

Take a system  $(\mathcal{X},\mathcal B,\mu,f)$, where $\mathcal{X}$ is a Riemannian manifold, $\mathcal{B}$ is the Borel
$\sigma$-algebra, $f:\mathcal{X}\to\mathcal{X}$ is a measurable map
and $\mu$ an $f$-invariant probability measure.
Suppose that the time series $X_0, X_1,\ldots$ arises from such a system simply by evaluating a given  observable $\varphi:\mathcal{X}\to\mathbb{R}\cup\{\pm\infty\}$ along the orbits of the system, or in other words, the time evolution given by successive iterations by $f$:
\begin{equation}
\label{eq:def-stat-stoch-proc-DS} X_n=\varphi\circ f^n,\quad \mbox{for
each } n\in {\mathbb N}.
\end{equation}
Clearly, $X_0, X_1,\ldots$ defined in this way is not necessarily an independent sequence.  However, $f$-invariance of $\mu$ guarantees that this stochastic process is stationary.

In this paper the novelty of the approach resides in the fact that instead of considering observables $\varphi:\X\to\R$ achieving a global maximum at a single point $\zeta\in\X$, as in the great majority of the literature about EVL and HTS for dynamical systems, or at a finite number of points $\xi_1, \ldots, \xi_k\in\X$, with $k\in\N$, as in \cite{HNT12,AFFR16}, we assume that the maximum is achieved on a countable set $\mathcal M=\{\xi_i\}_{i\in\N_0}$, which is the closure
of a subset of the orbit of some chosen point $\zeta\in \mathcal{X}$. More precisely, for a  certain point $\zeta\in\X$, we have that  $\mathcal M:=\overline{\{f^{m_i}(\zeta):i\in\mathbb{N}\}}$. For simplicity, we assume that $\xi_i=f^{m_i}(\zeta)$, for each $i\in\N$ and the sequence $\{\xi_i\}_{i\in\N}$ has only one accumulation point that we denote by $\xi_0$. Hence, for each $i\in\N_0$ we have that $\xi_i$ is an isolated point of $\mathcal M$. However, the important facts are that $\mathcal M$ is closed and countable. If none of the points of $\mathcal M$ lay on the same orbit then the analysis would actually be much simpler and in case there are more accumulation points of the orbit of some chosen point $\zeta\in \mathcal{X}$, then each such accumulation point would have to analysed as we will do here for $\xi_0$. Having more general accumulation sets, such as Cantor sets, raises new technical challenges, which we are leaving for an already mentioned ongoing work. 

We assume that the observable $\varphi$ also
satisfies
\begin{equation}\label{def:observable}
\varphi(x)=h_i(\dist (x,\xi_i)),\; \forall x\in B_{\varepsilon_i}(\xi_i),\;\;\; i\in\{1,2, \ldots\}
\end{equation}
where $B_{\varepsilon_i}(\xi_i) \cap B_{\varepsilon_j}(\xi_j)=\emptyset$, for all $i\neq j$, $\dist$ denotes some metric in $\X$ and  the function
$h_i:[0,+\infty)\to \mathbb{R}\cup\{\infty\}$ is such that 0 is a global maximum ($h_i(0)$ may be $+\infty$),
 $h_i$ is a strictly decreasing continuous bijection $h_i : V \to W$ in a neighbourhood $V$ of 0; and has one of the following three types of behaviour:

\begin{enumerate}
  \item  Type 1: there exists some strictly positive function $g : W \to \mathbb{R}$ such that for all $y \in\mathbb{R} $
\begin{equation*}
\label{eq:type1}
\lim_{ s\to h_i(0)}\frac{h_i^{-1}(s+yg(s))}{h_i^{-1}(s)}= e^{-ˆ'y};
\end{equation*}
 \item  Type 2: $h_i(0) = +\infty$ and there exists $\beta>0$ such that for all $y > 0$
 \begin{equation*}
\label{eq:type2}
\lim_{s\to\infty}\frac{h_i^{-1}(sy)}{h_i^{-1}(s)}=y^{-\beta};
 \end{equation*}
  \item Type 3: $h_i(0) = D < +\infty$ and there exists $\gamma > 0$ such that for all $y > 0$
  \begin{equation*}
  \label{eq:type3}
  \lim_{s\to0}\frac{ h^{-1}_i(Dˆ-sy)}{h^{-1}_i(D-ˆ's)}= y^\gamma.
  \end{equation*}
\end{enumerate}

We assume, of course, that $h_i(0)=h_j(0)$ for all $i,j\in\N_0$. Now, at $\xi_0$, we may have different types of behaviour. We may have, for example that $\varphi$ is continuous at $\xi_0$ or not. We will see examples of application of both types. One particular case of study, in which we have continuity of $\varphi$ at $\xi_0,$ is when we take:
\begin{equation}
\label{eq:dist-case}
\varphi(x)=h(\dist(x,\mathcal M)),
\end{equation}
where $h$ is of one of the three types above and $\dist(x,\mathcal M)=\inf\{\dist(x,y):\, y\in\mathcal M\}$. Note that in this case we may take $h_i=h$ for all $i\in\N$.

In order to study the extremal behaviour of the systems, we consider the random variables  $M_1, M_2,\ldots$ given by
\begin{equation}
\label{eq:Mn-definition}
M_n=\max\{X_0,\ldots,X_{n-1}\}.
\end{equation}

We say that we have an \emph{Extreme Value Law} (EVL) for $M_n$ if there is a non-degenerate d.f.\ $H:\R\to[0,1]$ with $H(0)=0$ and,  for every $\tau>0$, there exists a sequence of levels $u_n=u_n(\tau)$, $n=1,2,\ldots$,  such that
\begin{equation}
\label{eq:un}
  n\p(X_0>u_n)\to \tau,\;\mbox{ as $n\to\infty$,}
\end{equation}
and for which the following holds:
\begin{equation}
\label{eq:EVL-law}
\p(M_n\leq u_n)\to \bar H(\tau)=1-H(\tau),\;\mbox{ as $n\to\infty$.}
\end{equation}
where the convergence is meant at the continuity points of $H(\tau)$.

As described in \cite{F13,LFFF16} the study of the distributional limit for $M_n$ is tied to the occurrence of excedances, \ie the occurrence of events such as $\{X_j>u\}$, for some high threshold $u$, close to $u_F=\sup\{x: F(x)<1\}=\varphi(\xi_i)$, the right end of the support of the distribution function, $F$, of $X_0$. Note that the exceedances of $u$ correspond to hits of the orbits to the target set on $\X$ defined by:
\begin{equation*}
\label{def:U}
U(u):=\{x\in\mathcal{X}:\; \varphi(x)>u\}=\{X_0>u\}.
\end{equation*}

Observe that by the assumptions on $\varphi$ then $U(u)$ has possibly countably many connected components. Namely, we may write that
\begin{equation}
\label{eq:U-discontinuous}
U(u)=\bigcup_{j=1}^\infty B_{\varepsilon_j(u)}(\xi_j)\cup\{\xi_0\},
\end{equation}
where $B_{\varepsilon_j(u)}(\xi_j)$ denotes a ball of radius $\varepsilon_j(u)>0$ centred at $\xi_j$. Note that each $\eps_j(u)$ is determined by the function $h_j$ that applies to each $\xi_j$ in equation~\eqref{def:observable}.  These balls may overlap, such as when \eqref{eq:dist-case} holds,  in which case we can write that
\begin{equation}
\label{eq:U-continuous}
U(u)=\bigcup_{j=1}^{N(u)} B_{\varepsilon_j(u)}(\xi_j)\cup B_{\varepsilon_0(u)}(\xi_0),
\end{equation}
for some positive integer $N(u)$, which goes to $\infty$ as $u$ gets closer to $u_F$.

As usual, in order to avoid a non-degenerate limit for $M_n$ we  assume:
\begin{enumerate}
\item[{(R)}]
The quantity $\mu(U(u))$, as a function of $u$, varies continuously on a neighbourhood of $u_F$.
\end{enumerate}

\begin{remark}
\label{rem:R1}
Note that as long as the invariant measure has no atoms, then under the assumptions above on the observable we have that condition ($R$) is easily satisfied.
\end{remark}

\section{EVL with clustering caused by arbitrarily large periods}

In the study of extremes for dynamical systems, the appearance of clustering has been associated with the periodicity of a unique maximum of $\varphi$, as in \cite{FFT12}, or, more recently, with the fake periodicity borrowed by the existence of multiple correlated  maxima, as in \cite{AFFR16}. In both cases, the main idea to handle the short recurrence created by the periodic phenomena is to replace the events $U(u)$ by
\begin{align}
\label{eq:A-def}
\AA_{q}(u)&:=U(u)\cap\bigcap_{i=1}^{q}f^{-i}(U(u)^c)=\{X_0>u, X_1\leq u, \ldots, X_{q}\leq u\},
\end{align}
where we use the notation $A^c:=\mathcal X\setminus A$ for the complement of $A$ in $\X$, for all $A\in\mathcal B$, and $q$ plays the role of the ``period'' . In fact, in the case of a single maximum at a periodic point $\zeta$, then $q$ is actually the period of $\zeta$. This idea appeared first in \cite[Proposition~1]{FFT12} and was further elaborated in \cite[Proposition~2.7]{FFT15}.

However, the applications made so far always assumed that $q$ was a fixed positive integer. This is no longer compatible with the existence of  countable infinite number of maximal points of $\varphi$, which lay on the same orbit of some point. This means we need to adjust the conditions and arguments in order to consider the possibility of arbitrarily large $q$. Hence, we start by considering
%
the sequence $(q_n)_{n\in\N}$ to be such that
\begin{equation}
\label{eq:qn}
\lim_{n\rightarrow\infty}q_n=\infty\;\;\;\;\;\;\mbox{and}\;\;\;\;\;\; \lim_{n\rightarrow\infty}\frac{q_n}{n}=0.
\end{equation}

Let $(u_n)_{n\in\N}$ satisfy condition \eqref{eq:un} and set  $U_n:=U(u_n)$ and $\AA_{q_n,n}:=\AA_{q_n}(u_n)$, for all $n\in\N$. Also, let
\begin{equation}
\label{def:thetan}
\theta_n:=\frac{\mu\left(\AA_{q_n,n}\right)}{\mu(U_n)}.
\end{equation}

Let $B\in\B$ be an event. For some $s\geq0$ and $\ell\geq 0$, we define:
\begin{equation}
\label{eq:W-def}
\mathscr W_{s,\ell}(B)=\bigcap_{i=\lfloor s\rfloor}^{\lfloor s\rfloor+\max\{\lfloor\ell\rfloor-1,\ 0\}} f^{-i}(B^c).
\end{equation}
The notation $f^{-i}$ is used for the preimage by $f^i$. We will write $\mathscr W_{s,\ell}^c(B):=(\mathscr W_{s,\ell}(B))^c$.
Whenever is clear or unimportant which event $B\in\B$ applies, we will drop the $B$ and write just $\mathscr W_{s,\ell}$ or $\mathscr W_{s,\ell}^c$.
Observe that
\begin{equation}
\label{eq:EVL-HTS}
\mathscr W_{0,n}(U(u))=\{M_n\leq u\}.
\end{equation}

Here we adapt the two conditions $\D(u_n)$ and $\D'_q(u_n)$ of \cite{FFT15} for $q_n$ satisfying \eqref{eq:qn}.

\begin{condition}[$\D_{q_n}(u_n)$]\label{cond:D} We say that $\D_{q_n}(u_n)$ holds for the sequence $X_0,X_1,\ldots$ if for every  $\ell,t,n\in\N$
\begin{equation}\label{eq:D1}
\left|\mu\left( \AA_{q_n,n}\cap
 \mathscr W_{t,\ell}\left( \AA_{q_n,n}\right) \right)-\mu\left( \AA_{q_n,n}\right)
  \mu\left(\mathscr W_{0,\ell}\left( \AA_{q_n,n}\right)\right)\right|\leq \gamma(n,t),
\end{equation}
where $\gamma(n,t)$ is decreasing in $t$ for each $n$ and, there exists a sequence $(t_n)_{n\in\N}$ such that $t_n=o(n)$ and
$n\gamma(n,t_n)\to0$ when $n\rightarrow\infty$.
\end{condition}

Consider the sequence $(t_n)_{n\in\N}$, given by condition  $\D_{q_n}(u_n)$ and let $(k_n)_{n\in\N}$ be another sequence of integers such that
\begin{equation}
\label{eq:kn-sequence}
k_n\to\infty\quad \mbox{and}\quad  k_n t_n = o(n).
\end{equation}

\begin{condition}[$\D'_{q_n}(u_n)$]\label{cond:D'q} We say that $\D'_{q_n}(u_n)$
holds for the sequence $X_0,X_1,X_2,\ldots$ if there exists a sequence $(k_n)_{n\in\N}$ satisfying \eqref{eq:kn-sequence} and such that
\begin{equation}
\label{eq:D'rho-un}
\lim_{n\rightarrow\infty}\,n\sum_{j=q_n+1}^{\lfloor n/k_n\rfloor-1}\mu\left(  \AA_{q_n,n}\cap f^{-j}\left( \AA_{q_n,n}\right)
\right)=0.
\end{equation}
\end{condition}

We are now ready to state a result that gives us the existence of EVL under conditions $\D_{q_n}$ and $\D'_{q_n}$.

\begin{theorem}
\label{thm:qn}
Let $X_0, X_1, \ldots$ be a stationary stochastic process and $(u_n)_{n\in\N}$ a sequence satisfying \eqref{eq:un}, for some $\tau>0$. Assume that conditions $\D_{q_n}(u_n)$ and $\D_q'(u_n)$ hold for some sequence $(q_n)_{n\in\N_0}$ satisfying \eqref{eq:qn}, and sequences $(t_n)_{n\in\N}$ and $(k_n)_{n\in\N}$ as in the statement of those conditions.  Moreover assume that the limit $\lim_{n\to\infty} \theta_n=:\theta$ exists. Then
$$
\lim_{n\to\infty}\mu(M_n\leq u_n)=\lim_{n\to\infty}\mu(\mathscr W_{0,n}(U_n))= \lim_{n\to\infty}\mu(\mathscr W_{0,n}(\AA_{q_n,n}))=\e^{-\theta\tau}.
$$
\end{theorem}

\begin{proof}
The first equality follows trivially from \eqref{eq:EVL-HTS}. The second equality follows from applying \cite[Proposition~2.7]{FFT15} and stationarity to obtain that
$$
\left|\mu(\mathscr W_{0,n}(U_n))-\mu(\mathscr W_{0,n}(\AA_{q_n,n}))\right|\leq q_n\mu(U_n\setminus\AA_{q_n,n})
$$
and then observe that the term on the right vanishes as $n\to\infty$ because of the defining properties of $q_n$ and $u_n$. For the third equality we only need to use \cite[Proposition~2.10]{FFT15} and adapt the proofs of \cite[Theorem~2.3 and Corollary~2.4]{FFT15}, by replacing $q$ by $q_n$, satisfying \eqref{eq:qn}, to obtain:
\begin{align}
\big|\mu(\mathscr W_{0,n}(\AA_{q_n,n}))-\e^{-\theta\tau}\big|&\leq  C\Bigg[k_nt_n\frac{\tau}n+n\gamma(n,t_n)+n\sum_{j=q_n+1}^{\lfloor n/k_n \rfloor} \mu\left(  \AA_{q_n,n}\cap f^{-j}\left(  \AA_{q_n,n}\right)\right)\nonumber\\
&\quad+ \e^{-\theta\tau}\left(\left|\tau-n\mu\left( U_n\right)\right|+\frac {\tau^2}{k_n}+\left|\theta_n-\theta\right|\tau\right)\Bigg],
\end{align}
for some $C>0$ and then use the properties of $k_n$, $t_n$, $u_n$, plus the convergence of $\theta_n$ to $\theta$ and the new conditions $\D_{q_n}(u_n)$ and $\D'_{q_n}(u_n)$ in order to show that all terms on the right vanish as $n\to\infty$.
\end{proof}

\section{Applications to systems with countable  maximal sets}

\subsection{Assumptions on the system and examples of application}

We assume that the system admits a first return time induced map with decay of correlations against $L^1$ observables. In order to clarify what is meant by the latter we define:
\begin{definition}[Decay of correlations]
\label{def:dc}
Let \( \mathcal C_{1}, \mathcal C_{2} \) denote Banach spaces of real valued measurable functions defined on \( \X \).
We denote the \emph{correlation} of non-zero functions $\phi\in \mathcal C_{1}$ and  \( \psi\in \mathcal C_{2} \) w.r.t.\ a measure $\p$ as
\[
\cv_\p(\phi,\psi,n):=\frac{1}{\|\phi\|_{\mathcal C_{1}}\|\psi\|_{\mathcal C_{2}}}
\left|\int \phi\, (\psi\circ f^n)\, \dif\p-\int  \phi\, \dif\p\int
\psi\, \dif\p\right|.
\]

We say that we have \emph{decay
of correlations}, w.r.t.\ the measure $\p$, for observables in $\mathcal C_1$ \emph{against}
observables in $\mathcal C_2$ if, for every $\phi\in\mathcal C_1$ and every
$\psi\in\mathcal C_2$ we have
 $$\cv_\p(\phi,\psi,n)\to 0,\quad\text{ as $n\to\infty$.}$$
  \end{definition}

We say that we have \emph{decay of correlations against $L^1$
observables} whenever  this holds for $\mathcal C_2=L^1(\p)$  and
$\|\psi\|_{\mathcal C_{2}}=\|\psi\|_1=\int |\psi|\,\dif\p$.

If a system already has decay of correlations against $L^1$ observables, then by taking the whole set $\X$ as the base for the first return time induced map, which coincides with the original system, then the assumption we impose on the system is trivially satisfied. Examples of systems with such property include:
\begin{itemize}

\item Uniformly expanding maps on the circle/interval (see \cite{BG97});

\item Markov maps (see \cite{BG97});

\item Piecewise expanding maps of the interval with countably many branches like Rychlik maps (see \cite{R83});

\item Higher dimensional piecewise expanding maps studied by Saussol in \cite{S00}.

\end{itemize}

\begin{remark}
\label{rem:C-space}
In the first three examples above the Banach space $\mathcal C_1$ for the decay of correlations can be taken as the space of functions of bounded variation. In the fourth example  the Banach space $\mathcal C_1$ is the space of functions with finite quasi-H\"older norm studied in \cite{S00}. We refer to \cite{BG97,S00} or \cite{AFV15} for precise definitions but mention that if $J\subset \R$ is an interval then $\I_J$ is of bounded variation and its BV-norm is equal to 2, \ie $\|\I_J\|_{BV}=2$ and if $A$ denotes a ball or an annulus then $\I_A$ has a finite quasi-H\"older norm.
\end{remark}

Although the examples above are all in some sense uniformly hyperbolic, we can consider non-uniformly hyperbolic systems, such as intermittent maps, which admit a `nice' first return time induced map over some subset $Y\subset \X$, called the base of the induced map. To be more precise, consider the usual original system as $f:\X\to \X$ with an ergodic $f$-invariant probability measure $\mu$, choose a subset $Y\subset \X$ and consider $F_Y:Y\to Y$ to be the first return map $f^{r_Y}$ to $Y$ (note that $F$ may be undefined at a zero Lebesgue measure set of points which do not return to $Y$, but most of these points are not important, so we will abuse notation here).  Let $\mu_Y(\cdot)=\frac{\mu(\cdot \cap Y)}{\mu(Y)}$ be the conditional measure on $Y$.  By Kac's Theorem $\mu_Y$ is $F_Y$-invariant.

From  \cite{BSTV03,HayWinZwe14}, we know that the Hitting Times Statistics (which can be put in terms of EVL by \cite{FFT10,FFT11}) of the first return induced system coincide with that of the original system. So, as long as the maximal set $\mathcal M$ is contained in the base of the induced system, $Y$, then the induced and original system share the same EVL. Hence, in order to cover all these examples of systems with `nice' first return time induced maps we are reduced to proving the existence of EVL for systems with decay of correlations against $L^1$ observables.
This fact motivates the following:
\vspace{0.5cm}

\noindent \textbf{Assumption A}
\emph{Let $f:\X\to\X$ be a system with summable decay of correlations against $L^1$ observables, \ie for all $\varphi\in\mathcal C_1$ and $\psi\in L^1$, then $\cv(\varphi,\psi,n)\leq \rho_n$, with $\sum_{n\geq\N}\rho_n<\infty$. }
\vspace{0.5cm}

Among the examples of systems with these `nice' induced maps we mention  the \emph{Manneville-Pomeau} (MP) map equipped with an absolutely continuous invariant probability measure (see for example \cite{LSV99, BSTV03}) and Misiurewicz quadratic maps (see \cite{MS93}).  

\subsection{Main results}

\begin{theorem}
\label{thm:continuous}
Assume that $f:\X\to\X$ satisfies Assumption A. Let $X_0, X_1,\ldots$ be given by \eqref{eq:def-stat-stoch-proc-DS}, where $\varphi$ achieves a global maximum on the compact set $\mathcal M=\{\xi_i\}_{i\in\N_0}$, where $\xi_i=f^{m_i}(\zeta)$, for some $\zeta\in\X$, and $\xi_0$ is the only accumulation point of the sequence $\{\xi_i\}_{i\in\N}$. Let $(u_n)_{n\in\N}$ be as sequence of thresholds as in \eqref{eq:un}. Assume further that $\varphi$ is continuous at $\xi_0$ (as when \eqref{eq:dist-case} holds) and there exist $N(n)\in\N$ and $\eps_0(n),\ldots, \eps_{N(n)}(n)$ so that  $\lim_{n\to\infty}N(n)=\infty$, $\lim_{n\to\infty}N(n)/n=0$ and
$$
U_n=U(u_n)=\bigcup_{j=1}^{N(n)} B_{\eps_j(n)}(\xi_j)\cup B_{\eps_0(n)}(\xi_0),
$$
where for each $i\neq j$, we have $B_{\eps_i(n)}(\xi_i)\cap B_{\eps_j(n)}(\xi_j)=\emptyset$, for sufficiently large $n$.
Let $q_n=N(n)$, set $\AA_{q_n,n}:=\AA_{q_n}(u_n)$ as given by \eqref{eq:A-def}. If the conditions:
\begin{enumerate}
\item $\displaystyle \lim_{n\to\infty}\|\I_{\AA_{q_n,n}}\|_{\mathcal C_1}n\rho_{t_n}=0$, for some sequence $(t_n)_{n\in\N}$ such that $t_n=o(n)$
\item $\displaystyle \lim_{n\to\infty}\|\I_{\AA_{q_n,n}}\|_{\mathcal C_1}\sum_{j=N(n)}^\infty \rho_j=0$
\end{enumerate}
hold and the limit $\lim_{n\to\infty}\theta_n:=\theta$ exists, where $\theta_n$ is given by \eqref{def:thetan}, then we have that
$$
\lim_{n\to\infty}\mu(M_n\leq u_n)=\e^{-\theta\tau}.
$$
\end{theorem}

\begin{proof}
By Theorem~\ref{thm:qn}, we need to check that $X_0, X_1, \ldots$ satisfies conditions $\D_{q_n}$ and $\D'_{q_n}$.

\noindent\textbf{Verification of condition} $\D_{q_n}(u_n)$.

Taking  $\phi=\I_{ \AA_{{q_n,n}}}$ and  $\psi=\I_{ \mathscr{W}_{t,\ell}(\AA_{q_n,n})}$  in Definition \ref{def:dc}, there exists
$C>0$  , so that for any positive numbers $\ell$ and $t$ we have

\begin{align*}
  |\mu(\AA_{q_n,n}\cap \mathscr{W}_{t,\ell}(\AA_{q_n,n}))
                &-\mu(\AA_{q_n,n})\mu(\mathscr{W}_{0,\ell}(\AA_{q_n,n}))| \\
                &= \left|\int_{\mathcal{X}}\I_{ \AA_{q_n,n}}\cdot (\I_{ \mathscr{W}_{0,\ell}(\AA_{q_n,n})}\circ f^t)d\mu-\int_{\mathcal{X}}\I_{ \AA_{q_n,n}}d\mu \int_{\mathcal{X}}\I_{ \mathscr{W}_{0,\ell}(\AA_{q_n,n})}d\mu \right| \\
                &\leq C\|\I_{\AA_{q_n,n}}\|_{\mathcal C_1}\rho(t).
\end{align*}
Then, Condition $\D_{q_n}(u_n)$ follows if there exists  a sequence $(t_n)_{n\in\NN}$ such that $t_n=o(n)$ and
$\displaystyle\lim_{n\to\infty} \|\I_{\AA_{q_n,n}}\|_{\mathcal C_1}n\rho_{t_n}=0$, which is precisely the content of hypothesis (1).

\noindent\textbf{Verification of condition}  $\D_{q_n,}'(u_n)$
Taking $\phi=\psi=\I_{\AA_{q_n,n}}$  in Definition \ref{def:dc} we obtain
\begin{align}
\label{eq:estimate1}
\mu\left( \AA_{q_n,n}\cap f^{-j}( \AA_{q_n,n})\right)
            &=\int_Y \phi\cdot (\phi\circ f^{j})d\mu
            \leq \left(\mu( \AA_{q_n,n})\right)^2+ \left\| \I_{ \AA_{q_n,n}}\right\|_{\mathcal C_1} \mu\left( \AA_{q_n,n}\right) \rho_j.
\end{align}
Let $t_n$ be as above and take $(k_n)_{n\in\N}$ as in \eqref{eq:kn-sequence}. Recalling that  $\lim_{n\to\infty}n\mu(U_n)=\tau$ it follows that
\begin{align*}
n\sum_{j=q_n+1}^{\lfloor n/k_n \rfloor}\mu\left( \AA_{q_n,n}\cap f^{-j}( \AA_{q_n,n})\right)
                & = n\sum_{j=N(n)+1}^{\lfloor n/k_n \rfloor} \mu\left( \AA_{q_n,n}\cap f^{-j}( \AA_{q_n,n})\right)\\
                & \le n\big\lfloor\tfrac {n}{k_n}\big\rfloor\mu\left( \AA_{q_n,n}\right)^2 +n\left\| \I_{ \AA_{q_n,n}}\right\|_{\mathcal C_1} \mu\left( \AA_{q_n,n}\right) \sum_{j=N(n)+1}^{\lfloor n/k_n \rfloor}\rho_j\\
                &\le \frac{\left(n\mu( \AA_{q_n,n})\right)^2}{k_n} +n\left\| \I_{ \AA_{q_n,n}}\right\|_{\mathcal C_1} \mu\left( \AA_{q_n,n}\right) \sum_{j=N(n)}^{\infty}\rho_j\\
                &\leq \frac{\tau^2}{k_n}+\tau\left\| \I_{\AA_{q_n,n}}\right\|_{\mathcal C_1} \sum_{j=N(n)}^{\infty}\rho_j \xrightarrow[n\to\infty]{}0,
\end{align*}
by choice of $k_n$ and hypothesis (2).
\end{proof}

We observe that in the previous theorem since for each $n$ the number of connected components of both $U_n$ and $\AA_{q_n,n}$ is finite then the function $\I_{\AA_{q_n,n}}$ in principle belongs to the Banach space $\mathcal C_1$, when $\mathcal C_1$ is the space of functions of bounded variation or with finite quasi-H\"older norm considered in \cite{S00}. However, if the observable $\varphi$ is discontinuous at $0$ then both $U_n$ and $\AA_{q_n,n}$ may have an infinite number of connected components, which means that $\I_{\AA_{q_n,n}}$ does not belong to any of the Banach spaces mentioned. In such cases we must be more careful and introduce some suitable truncated versions of $U_n$ and $\AA_{q_n,n}$ as we will see in the following theorem.

\begin{theorem}
\label{thm:discontinuous}
Assume that $f:\X\to\X$ satisfies Assumption A. Let $X_0, X_1,\ldots$ be given by \eqref{eq:def-stat-stoch-proc-DS}, where $\varphi$ achieves a global maximum on the compact set $\mathcal M=\{\xi_i\}_{i\in\N_0}$, where $\xi_i=f^{m_i}(\zeta)$, for some $\zeta\in\X$, and $\xi_0$ is the only accumulation point of the sequence $\{\xi_i\}_{i\in\N}$. Let $(u_n)_{n\in\N}$ be as sequence of thresholds as in \eqref{eq:un}. Assume further that $\varphi$ is discontinuous at $\xi_0$ and there exist $\eps_1(n), \eps_2(n),\ldots,  $ such that
$$
U_n=U(u_n)=\bigcup_{j=1}^{\infty} B_{\eps_j(n)}(\xi_j)\cup \{\xi_0\}.
$$
Moreover, assume that and there exists $N(n)\in\N$ such that  $\lim_{n\to\infty}N(n)=\infty$, $\lim_{n\to\infty}\frac{N(n)}{n}=0$ and
$$
\lim_{n\to\infty}\frac{\mu(U_n\setminus \tilde U_n)}{\mu(U_n)}=0, \quad \text{where}\quad\tilde U_n=\bigcup_{j=1}^{N(n)} B_{\eps_j(n)}(\xi_j).
$$
Let $q_n=N(n)$, set $\tilde \AA_{q_n,n}:=\tilde U_n\cap\bigcap_{i=1}^{q_n}f^{-i}(\tilde U_n^c)$. If the conditions:
\begin{enumerate}
\item $\displaystyle \lim_{n\to\infty}\|\I_{\tilde \AA_{q_n,n}}\|_{\mathcal C_1}n\rho_{t_n}=0$, for some sequence $(t_n)_{n\in\N}$ such that $t_n=o(n)$
\item $\displaystyle \lim_{n\to\infty}\|\I_{\tilde \AA_{q_n,n}}\|_{\mathcal C_1}\sum_{j=N(n)}^\infty \rho_j=0$
\end{enumerate}
hold and the limit $\lim_{n\to\infty}\theta_n:=\theta$ exists, where $\theta_n=\frac{\mu(\tilde \AA_{q_n,n})}{\mu(\tilde U_n)}$, then we have that
$$
\lim_{n\to\infty}\mu(M_n\leq u_n)=\e^{-\theta\tau}.
$$
\end{theorem}

\begin{proof}
Using the exact same argument as in the proof of the previous theorem one can check that conditions $\D_{q_n}$ and $\D'_{q_n}$ hold, here, when the role of $\AA_{q_n,n}$ is replaced by that of $\tilde\AA_{q_n,n}$. Then an application of the second and third equalities of the conclusion of Theorem~\ref{thm:qn} allows us to obtain that $\lim_{n\to\infty} \mu(\mathscr W_{0,n}(\tilde U_n))=\e^{-\theta\tau}$. The missing step is to show that $$\lim_{n\to\infty}\mu(M_n\leq u_n)=\lim_{n\to\infty} \mu(\mathscr W_{0,n}(\tilde U_n)).$$
To see this we observe first that, by \eqref{eq:EVL-HTS}, we can replace $\{M_n\leq u_n\}=\mathscr W_{0,n}(U_n)$ and then, by stationarity,
\begin{align*}
\mu(\mathscr W_{0,n}(\tilde U_n)\setminus \mathscr W_{0,n}(U_n))&\leq \sum_{i=0}^{n-1}\mu(f^{-i}(U_n\setminus \tilde U_n))=n\mu(U_n\setminus \tilde U_n).
\end{align*}
Recalling that $\lim_{n\to\infty}n\mu(U_n)=\tau$ and by hypothesis $\lim_{n\to\infty}\frac{\mu(U_n\setminus \tilde U_n)}{\mu(U_n)}=0$, then we realise that the last term on the right vanishes as $n\to\infty$ and the result follows.
\end{proof}

\begin{example}\label{ex1}
  Let $(\mathbb{S}^1,f,\l)$ be the system where  $f(x)=3x \mod1$ and $\l$ is the Lebesgue measure. Let $(\xi_j)_{j\in\NN}$
be the sequence defined as $\xi_j=f^{3^j}(z)$, where $z=\sum_{i=1}^{\infty}\left(\frac{1}{3}\right)^{3^i}$.
Letting $\mathcal{M}=\{0,z,\xi_1,\xi_2,\dots\}$, we notice that $\mathcal{M}$ is closed and $0$ is the unique accumulation point
of $\mathcal{M}$ with $\lim_{j\to\infty}\xi_j=0$.

 Notice that $z>\xi_1>\xi_2>\dots$. We set  $I_0:=\left[\frac{z+\xi_1}{2},z\right]$, $I_1:=\left[\frac{\xi_2+\xi_1}{2},\frac{\xi_1+z}{2}\right]$ and $I_j:=\left[\frac{\xi_{j+1}+\xi_j}{2},\frac{\xi_{j-1}+\xi_j}{2}\right]$, j=$2,3,\ldots$ 
 
 Consider the observable $\varphi:\mathbb S^1\to\mathbb{R}$ given by
$$
\varphi|_{[\frac{z+\xi_{1}}{2},z]}(x)=\frac{2x}{z-\xi_1}-\frac{z+\xi_1}{z-\xi_1},\;\;
\varphi|_{[z,1)}(x)=0,\; \varphi(1)=1,
$$
and
$$
\displaystyle \varphi|_{I_j}(x)=\left\{
                  \begin{array}{ll}
                    \frac{2x}{\xi_j-\xi_{j+1}}-\frac{\xi_{j+1}+\xi_j}{\xi_j-\xi_{j+1}}, & \hbox{ for }x\in [\frac{\xi_{j+1}+\xi_j}{2},\xi_j] \\
                     -\frac{2x}{\xi_{j-1}-\xi_j}+\frac{\xi_{j-1}+\xi_j}{\xi_{j-1}-\xi_j}, & \hbox{ for }x\in [\xi_j,\frac{\xi_{j-1}+\xi_j}{2}]
                  \end{array}
                \right.,
 $$
for $j\in \mathbb N$ if we set $\xi_0:=z.$
\begin{figure}
\includegraphics[height=2cm,width=15cm]{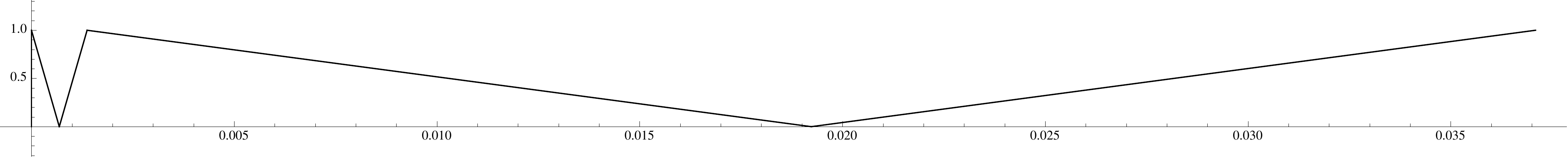}
\caption{The picture is a graph of the observable of Example~\ref{ex1} in the interval $[0,z]$.}
\end{figure}

 The maximum of $\varphi$ is $1$ and occurs on $\mathcal{M}$.
Given $(u_n)_{n\in\NN}$ a sequence of tresholds as in \eqref{eq:un}, 
$$
\varphi|_{I_0}(x)>u_n\Leftrightarrow x\in \left(\frac{\xi_{1}+z}{2}+\frac{z-\xi_1}{2}u_n,z\right): =I_{0,n},
$$

$$
\varphi|_{I_1}(x)>u_n\Leftrightarrow x\in \left(\frac{\xi_{2}+\xi_1}{2}+\frac{\xi_{1}-\xi_2}{2}u_n,\frac{\xi_{1}+z}{2}-\frac{z-\xi_1}{2}u_n\right): =I_{1,n},
$$
and, for $j\geq 2$,
$$
\varphi|_{I_j}(x)>u_n\Leftrightarrow x\in\left(\frac{\xi_{j+1}+\xi_j}{2}+\frac{\xi_{j}-\xi_{j+1}}{2}u_n,\frac{\xi_{j}+\xi_{j-1}}{2}-\frac{\xi_{j-1}-\xi_j}{2}u_n\right): =I_{j,n}.
$$

 As in the hypotheses of Theorem \ref{thm:discontinuous}\footnote{We note that the sets $I_j, I_{j,n}$ are not symmetric with respect to the centre $\xi_j$ but by changing the metric we can still identify the sets $I_j, I_{j,n}$ as balls around $\xi_j$.},
$U_n=\bigcup_{j=0}^{\infty}I_{j,n}\cup \{0\}$, with $|I_{j,n}|=(1-u_n)|I_j|$. 
For each $n\in \NN$, set $N(n)=\lceil\log_3(\log_3n)+1\rceil$. Let $\tilde{U}_n=\bigcup_{j\leq N(n)}I_{j,n}$
and  notice that, as 

\begin{equation}
\label{est_xi_j}
\left(\frac{1}{3}\right)^{2\cdot3^j}\leq \xi_j\leq \left(\frac{1}{3}\right)^{2\cdot3^j}
+\frac{9}{8}\left(\frac{1}{3}\right)^{8\cdot3^j},
\end{equation}
then

\begin{equation}
\label{est_med_xi_j}
\frac{1}{2}\left(\frac{1}{3}\right)^{2\cdot3^{j-1}}\leq \frac{\xi_j+\xi_{j-1}}{2}\leq \left(\frac{1}{3}\right)^{2\cdot3^{j-1}}  
\end{equation}

So, based on the last inequality, we obtain

\begin{align*}
\bigcup_{j>N(n)}I_{j,n}\subset \left[0,\frac{1}{n^2}\right)
  \Rightarrow \mu\left(\bigcup_{j>N(n)}I_{j,n}\right)\leq\frac{1}{n^2}
                                                 \Rightarrow \frac{\mu\left(U_n\backslash \tilde U_n\right)}{\mu(U_n)}=\frac{\mu\left(\bigcup_{j>N(n)}I_{j,n}\right)}{\mu(U_n)}
                                                 \lesssim\frac{1}{n },
\end{align*}
since $\mu(U_n) \sim \tau\slash n$.
Letting  $q_n=N(n)$ and $\tilde \AA_{q_n,n}$ as in Theorem \ref{thm:discontinuous}, we have that
$$
\left\| \I_{\tilde\AA_{q_n,n}}\right\|_{BV}\leq 4 N(n)+1.
$$

By \cite{BG97}, the system $(\mathbb{S}^1,f,\l)$ has exponential decay of correlations against $L^\infty$ observables
with $\mathcal C_1=BV$, \ie there exist $C>0$ and $r\in (0,1)$
so that for $\phi\in BV$ and $\psi\in L^\infty$
$$
\cv_\l(\phi,\psi,n)\leq C r^n.
$$
For $t_n=\sqrt{n}$ we get that
\begin{align*}
  \lim_{n\to\infty}\left\| \I_{\tilde\AA_{q_n,n}}\right\|_{BV} n \ r^{t_n} & \leq \lim_{n\to\infty} \frac{(4 N(n)+1)}{n}\frac{n^2}{ \e^{\sqrt{n}\log \frac{1}{r}}}=0,
\end{align*}
and for some $C'>0$,
\begin{align*}
 \lim_{n\to\infty}\left\| \I_{\tilde\AA_{q_n,n}}\right\|_{BV} \sum_{j=N(n)}^{\infty}r^j
                \leq C' \lim_{n\to\infty}(4 N(n)+1) r^{N(n)}=0.
\end{align*}
Now we observe that
\begin{align*}
\mu(\tilde U_n)=(1-u_n)\left(z-\frac{\xi_{N(n)}+\xi_{N(n)+1}}{2}\right)
\end{align*}
and, letting $J(n)=\max\{j:3^j\leq N(n)\}$, 

\begin{align*}
\mu\left(\tilde\AA_{q_n,n}\right) & = (1-u_n)\left(    \frac{z-\xi_1}{2}-3^{-3}\frac{\xi_1-\xi_2}{2}           +\sum_{j=1}^{J(n)}\left(|I_j|-3^{-3^{(j+1)}+3^j}|I_{j+1}|\right)
+
                                   \sum_{j=J(n)+1}^{N(n)}|I_j|\right)      \\
                      & =  (1-u_n)            \left(    \frac{z-\xi_1}{2}         +\sum_{j=1}^{J(n)}|I_j|
                      -
                                   3^{-3}\frac{\xi_1-\xi_2}{2} -
                      \sum_{j=1}^{J(n)}\left(3^{-3^{(j+1)}+3^j}|I_{j+1}|\right)
+
 \sum_{j=J(n)+1}^{N(n)}|I_j|\right)    \\
                   & =  (1-u_n)            \left(  \sum_{j=0}^{J(n)}|I_j|
                      -
                                   3^{-3}\frac{\xi_1-\xi_2}{2} -
                      \sum_{j=1}^{J(n)}\left(3^{-3^{(j+1)}+3^j}|I_{j+1}|\right)
+
 \sum_{j=J(n)+1}^{N(n)}|I_j|\right).                                                                                 
\end{align*}
\\

Note that, since  $\lim_{j\to\infty}\xi_j=0$, we have that
\begin{align*}
\displaystyle\lim_{n\to\infty}\frac{(1-u_n)\sum_{j=J(n)+1}^{N(n)}|I_j|}{\mu(\tilde U_n)}
                =\lim_{n\to\infty}\frac{\frac{\xi_{J(n)}+\xi_{J(n)+1}-(\xi_{N(n)}+\xi_{N(n)+1})}{2}}{z-\frac{\xi_{N(n)}+\xi_{N(n)+1}}{2}}=0.
\end{align*}
\\

Moreover
\[
\displaystyle\lim_{n\to\infty}\frac{(1-u_n)\sum_{j=0}^{J(n)}|I_j|}{\mu(\tilde U_n)}
                =\lim_{n\to\infty}\frac{ z-\frac{\xi_{J(n)}+\xi_{J(n)+1}}{2}}{z-\frac{\xi_{N(n)}+\xi_{N(n)+1}}{2}}=1.
\]
\\

Consequently, by Theorem by \ref{thm:discontinuous}, the extremal index is given by
\begin{align*}
 \theta=& \lim_{n\to\infty}\theta_n = \lim_{n\to\infty}\frac{\mu\left(\tilde\AA_{q_n,n}\right)}{\mu(\tilde U_n)}
                            =1-  \lim_{n\to\infty}\frac{(1-u_n)\left( 3^{-3}\frac{\xi_1-\xi_2}{2} +
                      \sum_{j=1}^{J(n)}\left(3^{-3^{(j+1)}+3^j}|I_{j+1}|\right)\right)}{(1-u_n)\left(z-\frac{\xi_{N(n)}+\xi_{N(n)+1}}{2}\right)}
\\
                      & =  1-\frac{1}{z}\left( 3^{-3}\cdot\frac{\xi_1-\xi_2}{2}+\lim_{n\to\infty}     \sum_{j=1}^{J(n)}\left(3^{-3^{(j+1)}+3^j}\cdot\frac{\xi_{j}-\xi_{j+2}}{2}\right)  \right)
\end{align*}
and so, $\lim_{n\to\infty} \mu(M_n\leq u_n)=\e^{-\theta \tau}$. A numerical approximation for $\theta$ to the 12th digit gives:
$$
\theta\approx 0.999289701946552
$$
\end{example}

The value of the EI is very close to 1 because it takes a very long time for the maximal set to recur to itself and the major contributions for reducing the EI come from the points that recur faster.

\begin{example}
\label{ex2}
In this example we consider the same system and maximal set, but we change the potential.
As a consequence of that change we get that the set $U_n$ is given by a finite union of
open intervals and, in this case, it is not necessary to build approximations for the set $U_n$.

Let $(\mathbb{S}^1,f,\l)$ and $\mathcal{M}$ be the system and the set
 defined in Example \ref{ex1}. Define  $\varphi:\mathbb{S}^1\to \mathbb R\cup\{\infty\}$
 by $\varphi(x)=-\log d(x,\mathcal{M})$, which attains its maximum at $\mathcal M$ and is continuous
 in $0$, the only accumulation point of $\mathcal M$.

\begin{figure}
\includegraphics[width=7cm]{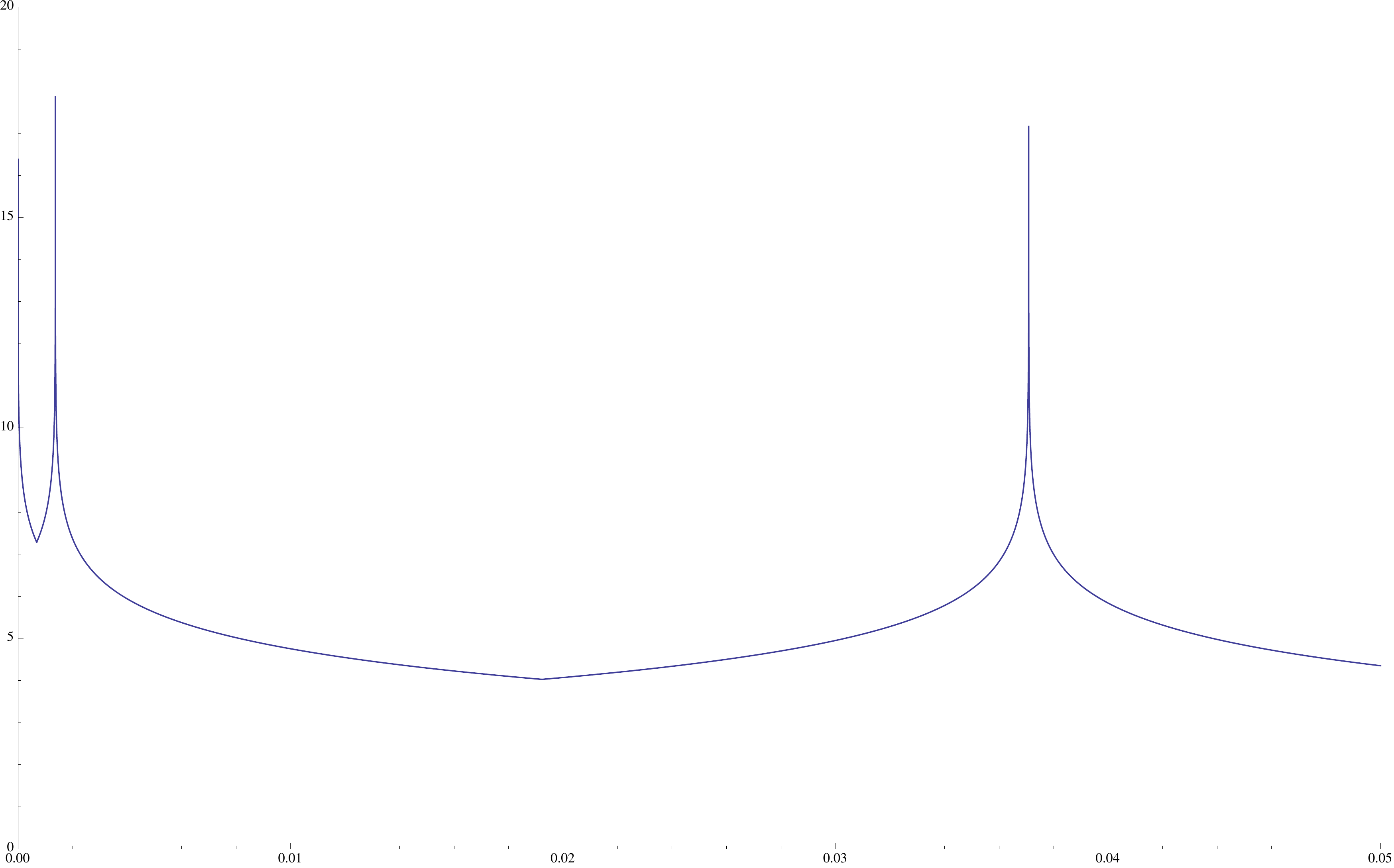}
\includegraphics[width=7cm]{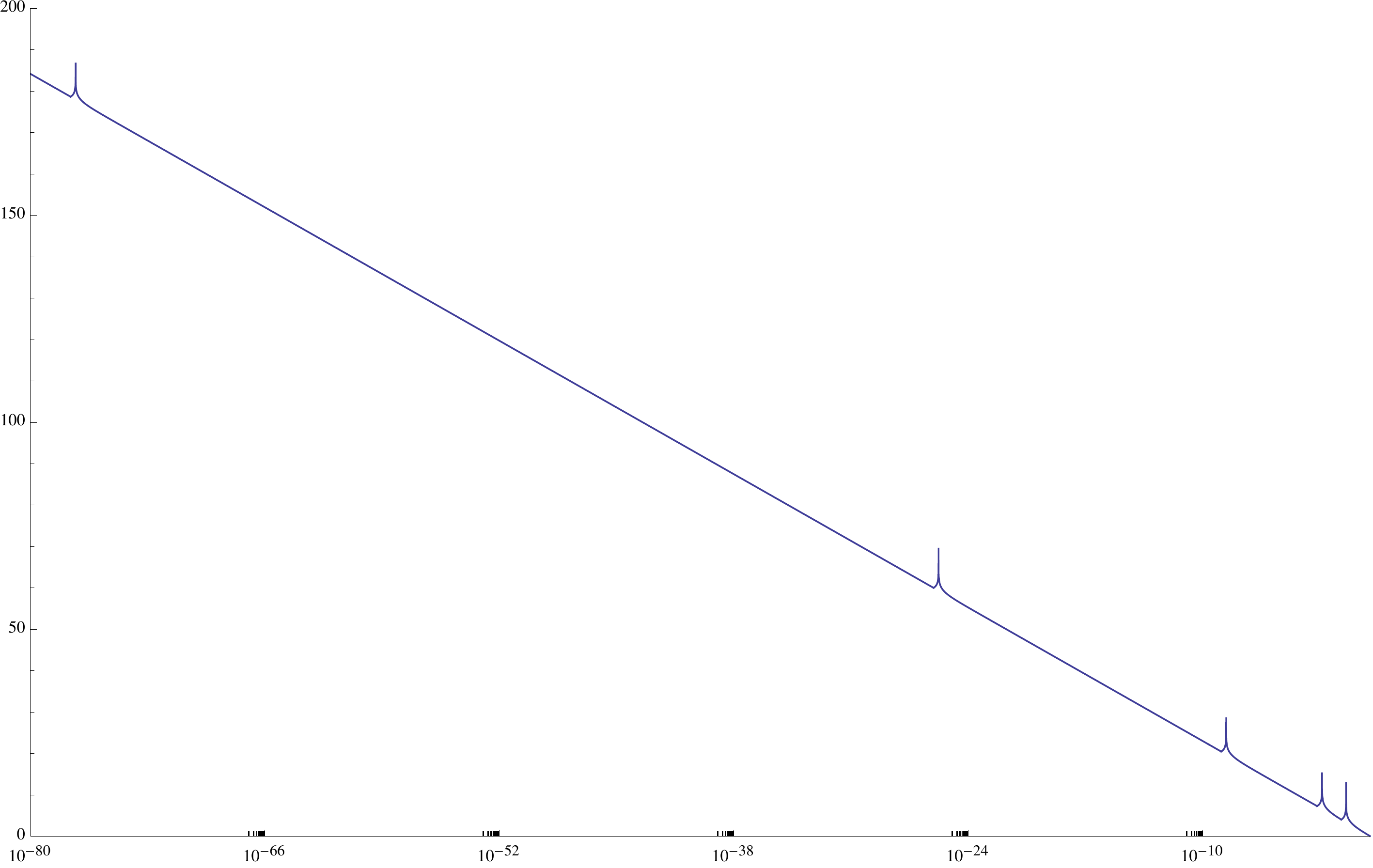}
\caption{The picture on the left is a graph of the observable of Example~\ref{ex2} and the picture on the right is the same graph with a logarithmic scale on the $x$-axis. The spikes correspond to the vertical asymptotes at $\xi_j$.}
\end{figure}

 In this example, we denote $I_j:=\left[\xi_{j+1},\xi_j\right]$, for $j=0,1,2,\ldots$, with $\xi_0:=z$,  and $\xi_j$, $j=1,2,\ldots$,  as in Example  \ref{ex1}. In each interval $I_j$ we have that
the minimum of $\varphi$ occurs in the medium point, \ie  in $\frac{\xi_{j+1}+\xi_j}{2}$.
So, if $\varphi\left(\frac{\xi_{j+1}+\xi_j}{2}\right)>u_n$, then $ \varphi|_{I_j}>u_n$.
Now we observe that $|\xi_{j}-\xi_{j+1}|>|\xi_{j+1}-\xi_{j+2}|$ for all $j=0,1,\dots$. Then
\begin{align*}
\varphi\left(\frac{\xi_{j+1}+\xi_{j+2}}{2}\right)>\varphi\left(\frac{\xi_{j+1}+\xi_{j}}{2}\right) \mbox{ for all $j=0,1,\dots$}
\end{align*}
 For each $n\in\NN$ let $N(n)=\max\{j\in\NN:\xi_j-\xi_{j+1}\geq 2\e^{-u_n}\}$.
So,
$$
 j> N(n)\Rightarrow\varphi\left(\frac{\xi_{j+1}+\xi_{j}}{2}\right)>u_n
$$
and $B_{\e^{-u_n}}(\xi_{j+1})\cap B_{\e^{-u_n}}(\xi_{j})=\emptyset$,
for $j\in\{1,\dots, N(n)\}$. Therefore,
$$
U_n=\left(0,\xi_{N(n)+1}+\e^{-u_n}\right)\cup \left(\cup_{j=1}^{N(n)}B_{\e^{-u_n}}(\xi_j)\right),
$$
and then $U_n$  has finitely many connected components, for each $n\in\NN$.
Let $q_n:=N(n)$ and consider $\AA_{q_n,n}=\left\{x\in U_n:f^{i}(x)\not\in U_n,\;i=1,2,\dots, q_n\right\}$.
By the definition of $N(n)$ and the sequence $(\xi_j)_{j\in\NN}$ we have that
$$
\xi_{N(n)+1}+\e^{-u_n}< \xi_{N(n)}\Rightarrow 3^{N(n)}(\xi_{N(n)+1}+\e^{-u_n})< 3^{N(n)}\xi_{N(n)}<\xi_{N(n)-1}.
$$
It follows that
$$
\left(\bigcup_{j=1}^{N(n)}f^j\left(\left(0,\xi_{N(n)+1}+\e^{-u_n}\right)\right)\right)\cap\left(\bigcup_{j=1}^{N(n)-2}B_{\e^{-u_n}}(\xi_j)\right)=\emptyset,
$$
 and then, $\AA_{q_n,n}\cap\left(0,\xi_{N(n)+1}+\e^{-u_n}\right)$ has at most $3$ connected components, which implies that  $\AA_{q_n,n}$ has at most $2N(n)+3$ and so,
  $\|\I_{\AA_{q_n,n}}\|_{BV}\leq 2(2N(n)+3)< 5N(n)$ .

Now we observe that
$\frac{\xi_{j}-\xi_{j+1}}{2}\leq \frac{\xi_j}{2}\leq\left(\frac{1}{3}\right)^{2\cdot3^j}$. So  $N(n)$ is less or equal to $\left\lceil\log_3\left(\frac{u_n}{2\ln3}\right)\right\rceil$ and
$
\lim_{n\to\infty}\frac{N(n)}{n}=0.
$
Moreover, for $t_n=\sqrt{n},$
\begin{align}\label{eq:dc1}
\nonumber\|\I_{\AA_{q_n,n}}\|_{BV}nr^{t_n}&\leq \left\lceil\log_3\left(\frac{u_n}{2\cdot\ln3}\right)\right\rceil^2 nr^{t_n},\\
\|\I_{\AA_{q_n,n}}\|_{BV}\sum_{j=N(n)}^\infty r^j      &\leq5\left\lceil\log_3\left(\frac{u_n}{2\cdot\ln3}\right)\right\rceil^2\sum_{j=N(n)}^\infty r^j.
\end{align}
By the definition of $u_n$ and by  \eqref{eq:dc1}  we have that $\lim_{n\to\infty}\|\I_{\AA_{q_n,n}}\|_{BV}nr^{t_n}=0$ and \linebreak $\lim_{n\to\infty}\|\I_{\AA_{q_n,n}}\|_{BV}\sum_{j=N(n)}^\infty r^j=0$. 

Now, let $J(n)=\max\{j:3^j\leq N(n)\}$ and notice that
\begin{align*}
\theta_n &=\frac{\mu\left(\AA_{q_n,n}\right)}{\mu(U_n)}=\frac{\mu\left(U_n\cap\AA_{q_n,n}\right)}{\mu(U_n)}=\frac{\mu\left(U_n\right)-\mu\left(U_n\setminus\AA_{q_n,n}\right)}{\mu(U_n)}\\
&=1-\frac{\mu\left(U_n\setminus\AA_{q_n,n}\right)}{\mu(U_n)}
\\
&=1-\frac{3^{-1}(\xi_{N(n)+1}+\e^{-u_n})+2\e^{-u_n}\sum_{j=0}^{J(n)}
3^{3^j-3^{j+1}}}{(\xi_{N(n)+1}+\e^{-u_n})+2\e^{-u_n}(N(n)+1)}.
\end{align*}
Note now that $\xi_{N(n)+1}\leq \e^{-u_n}$ and $\lim_{n\to\infty}\e^{-u_n}=0$. Moreover $J(n)\leq \log_3 N(n)$. So by Theorem \ref{thm:continuous}, the extremal index is given by $\theta =\lim_{n\to\infty}\theta_n=1$, that is,  $\lim_{n\to\infty} \mu(M_n\leq u_n)=\e^{- \tau}$.
\end{example}

In this case, we note that the set $U_n$ has an increasing number ($N(n)$) of connected components being that all of them have the same measure except for the utmost left one, which is at most 3 times larger. Again, as in the previous example, the maximal set takes a long time to recur, which means that the pieces that are extracted from $U_n$ to obtain $A_{q_n,n}$ become superexponentially small. Moreover, since all components have almost the same size, then the relative weight of the largest extractions, which occur in the fastest recurrent components, becomes more and more negligible when compared to the components that up to time $N(n)$ have still not suffered any extraction because they have not recurred, yet. Hence, in the end, the increasing weight of the components with no extractions leads to an EI equal to 1.  

\begin{remark}
We note that, in both examples, the accumulation point $\xi_0=0$ plays a secondary role on the computation  of the EI. In fact, we can easily conclude that if there was only one maximal point at $0$, which is a fixed point, then we would obtain an EI equal to $\theta=1-\frac12=\frac12$, by the formula in \cite[Theorem~3]{FFT12}, for example. This contrasts with the values obtained here for the EI.
\end{remark}

\section*{Acknowledgements}
DA was partially supported by the  grant  346300 for IMPAN from the Simons Foundation and the matching 2015-2019 Polish MNiSW fund. ACMF and JMF were partially supported by FCT projects FAPESP/19805/2014 and PTDC/MAT-CAL/3884/2014. FBR was supported by BREUDS, a Brazilian-European partnership of the FP7-PEOPLE-2012-IRSES program, with project number 318999, which is supported by an FP7 International International Research Staff Exchange Scheme (IRSES) grant of the European Union. All authors were partially supported by CMUP (UID/MAT/00144/2013), which is funded by FCT (Portugal) with national (MEC) and European structural funds through the programs FEDER, under the partnership agreement PT2020.

JMF would like to thank Mike Todd for careful reading and useful suggestions.

\bibliographystyle{amsalpha}
\bibliography{MultipleMaxima}

\newcommand{\etalchar}[1]{$^{#1}$}
\providecommand{\bysame}{\leavevmode\hbox to3em{\hrulefill}\thinspace}
\providecommand{\MR}{\relax\ifhmode\unskip\space\fi MR }
\providecommand{\MRhref}[2]{%
  \href{http://www.ams.org/mathscinet-getitem?mr=#1}{#2}
}
\providecommand{\href}[2]{#2}
\begin{thebibliography}{AFFR16}

\bibitem[Aba04]{A04}
Miguel Abadi, \emph{Sharp error terms and necessary conditions for exponential
  hitting times in mixing processes}, Ann. Probab. \textbf{32} (2004), no.~1A,
  243--264. \MR{MR2040782 (2004m:60042)}

\bibitem[AFFR16]{AFFR16}
Davide Azevedo, Ana Cristina~Moreira Freitas, Jorge~Milhazes Freitas, and
  Fagner~B. Rodrigues, \emph{Clustering of extreme events created by multiple
  correlated maxima}, Phys. D \textbf{315} (2016), 33--48. \MR{3426918}

\bibitem[AFV15]{AFV15}
Hale Ayta\c{c}, Jorge~Milhazes Freitas, and Sandro Vaienti, \emph{Laws of rare
  events for deterministic and random dynamical systems}, Trans. Amer. Math.
  Soc. \textbf{367} (2015), no.~11, 8229--8278. \MR{3391915}

\bibitem[BG97]{BG97}
Abraham Boyarsky and Pawel G{\'o}ra, \emph{Laws of chaos}, Probability and its
  Applications, Birkh\"auser Boston Inc., Boston, MA, 1997, Invariant measures
  and dynamical systems in one dimension. \MR{1461536 (99a:58102)}

\bibitem[BSTV03]{BSTV03}
H.~Bruin, B.~Saussol, S.~Troubetzkoy, and S.~Vaienti, \emph{Return time
  statistics via inducing}, Ergodic Theory Dynam. Systems \textbf{23} (2003),
  no.~4, 991--1013. \MR{MR1997964 (2005a:37004)}

\bibitem[CC13]{CC13}
J.-R. Chazottes and P.~Collet, \emph{Poisson approximation for the number of
  visits to balls in non-uniformly hyperbolic dynamical systems}, Ergodic
  Theory Dynam. Systems \textbf{33} (2013), no.~1, 49--80. \MR{3009103}

\bibitem[Col01]{C01}
P.~Collet, \emph{Statistics of closest return for some non-uniformly hyperbolic
  systems}, Ergodic Theory Dynam. Systems \textbf{21} (2001), no.~2, 401--420.
  \MR{MR1827111 (2002a:37038)}

\bibitem[FFT10]{FFT10}
Ana Cristina~Moreira Freitas, Jorge~Milhazes Freitas, and Mike Todd,
  \emph{Hitting time statistics and extreme value theory}, Probab. Theory
  Related Fields \textbf{147} (2010), no.~3-4, 675--710. \MR{2639719
  (2011g:37015)}

\bibitem[FFT11]{FFT11}
\bysame, \emph{Extreme value laws in dynamical systems for non-smooth
  observations}, J. Stat. Phys. \textbf{142} (2011), no.~1, 108--126.
  \MR{2749711 (2012a:60149)}

\bibitem[FFT12]{FFT12}
\bysame, \emph{The extremal index, hitting time statistics and periodicity},
  Adv. Math. \textbf{231} (2012), no.~5, 2626--2665. \MR{2970462}

\bibitem[FFT13]{FFT13}
\bysame, \emph{The compound {P}oisson limit ruling periodic extreme behaviour
  of non-uniformly hyperbolic dynamics}, Comm. Math. Phys. \textbf{321} (2013),
  no.~2, 483--527. \MR{3063917}

\bibitem[FFT15]{FFT15}
\bysame, \emph{Speed of convergence for laws of rare events and escape rates},
  Stochastic Process. Appl. \textbf{125} (2015), no.~4, 1653--1687.

\bibitem[FFTV15]{FFTV15}
Ana Cristina~Moreira Freitas, Jorge~Milhazes Freitas, Mike Todd, and Sandro
  Vaienti, \emph{Rare events for the manneville-pomeau map}, Preprint
  arXiv:1503.01372, March 2015.

\bibitem[FP12]{FP12}
Andrew Ferguson and Mark Pollicott, \emph{Escape rates for {G}ibbs measures},
  Ergod. Theory Dynam. Systems \textbf{32} (2012), 961--988.

\bibitem[Fre13]{F13}
Jorge~Milhazes Freitas, \emph{Extremal behaviour of chaotic dynamics}, Dyn.
  Syst. \textbf{28} (2013), no.~3, 302--332.

\bibitem[Hir93]{H93}
Masaki Hirata, \emph{Poisson law for {A}xiom {A} diffeomorphisms}, Ergodic
  Theory Dynam. Systems \textbf{13} (1993), no.~3, 533--556. \MR{MR1245828
  (94m:58137)}

\bibitem[HNT12]{HNT12}
Mark Holland, Matthew Nicol, and Andrei T{\"o}r{\"o}k, \emph{Extreme value
  theory for non-uniformly expanding dynamical systems}, Trans. Amer. Math.
  Soc. \textbf{364} (2012), no.~2, 661--688. \MR{2846347 (2012k:37064)}

\bibitem[HSV99]{HSV99}
Masaki Hirata, Beno{\^{\i}}t Saussol, and Sandro Vaienti, \emph{Statistics of
  return times: a general framework and new applications}, Comm. Math. Phys.
  \textbf{206} (1999), no.~1, 33--55. \MR{MR1736991 (2001c:37007)}

\bibitem[HV09]{HV09}
Nicolai Haydn and Sandro Vaienti, \emph{The compound {P}oisson distribution and
  return times in dynamical systems}, Probab. Theory Related Fields
  \textbf{144} (2009), no.~3-4, 517--542. \MR{MR2496441}

\bibitem[HWZ14]{HayWinZwe14}
Nicolai Haydn, Nicole Winterberg, and Roland Zweim{\"u}ller, \emph{Return-time
  statistics, hitting-time statistics and inducing}, Ergodic theory, open
  dynamics, and coherent structures, Springer Proc. Math. Stat., vol.~70,
  Springer, New York, 2014, pp.~217--227. \MR{3213501}

\bibitem[Kel12]{K12}
Gerhard Keller, \emph{Rare events, exponential hitting times and extremal
  indices via spectral perturbation}, Dyn. Syst. \textbf{27} (2012), no.~1,
  11--27. \MR{2903242}

\bibitem[KR14]{KR14}
Yuri Kifer and Ariel Rapaport, \emph{Poisson and compound {P}oisson
  approximations in conventional and nonconventional setups}, Probab. Theory
  Related Fields \textbf{160} (2014), no.~3-4, 797--831. \MR{3278921}

\bibitem[LFF{\etalchar{+}}16]{LFFF16}
Valerio Lucarini, Davide Faranda, Ana Cristina~Moreira Freitas, Jorge~Milhazes
  Freitas, Mark Holland, Tobias Kuna, Matthew Nicol, and Sandro Vaienti,
  \emph{Extremes and recurrence in dynamical systems}, Pure and Applied
  Mathematics: A Wiley Series of Texts, Monographs and Tracts, Wiley, Hoboken,
  NJ, 2016.

\bibitem[LSV99]{LSV99}
Carlangelo Liverani, Beno{\^{\i}}t Saussol, and Sandro Vaienti, \emph{A
  probabilistic approach to intermittency}, Ergodic Theory Dynam. Systems
  \textbf{19} (1999), no.~3, 671--685. \MR{MR1695915 (2000d:37029)}

\bibitem[MP15]{MP15}
Giorgio Mantica and Luca Perotti, \emph{Extreme value laws for fractal
  intensity functions in dynamical systems: Minkowski analysis}, Preprint
  arXiv:1512.07383, 2015.

\bibitem[MvS93]{MS93}
Welington~de Melo and Sebastian van Strien, \emph{One-dimensional dynamics},
  Ergebnisse der Mathematik und ihrer Grenzgebiete (3) [Results in Mathematics
  and Related Areas (3)], vol.~25, Springer-Verlag, Berlin, 1993. \MR{MR1239171
  (95a:58035)}

\bibitem[PS14]{PS14}
Fran\c{c}oise P\`ene and Beno\^{i}t Saussol, \emph{Poisson law for some
  nonuniformly hyperbolic dynamical systems with polynomial rate of mixing},
  Preprint arXiv:1401.3599, 2014.

\bibitem[Ryc83]{R83}
Marek Rychlik, \emph{Bounded variation and invariant measures}, Studia Math.
  \textbf{76} (1983), no.~1, 69--80. \MR{MR728198 (85h:28019)}

\bibitem[Sau00]{S00}
Beno{\^{\i}}t Saussol, \emph{Absolutely continuous invariant measures for
  multidimensional expanding maps}, Israel J. Math. \textbf{116} (2000),
  223--248. \MR{1759406 (2001e:37037)}

\end{thebibliography}

\end{document}